\documentclass[smallheadings,abstracton]{scrartcl}

\usepackage[OT2,T1]{fontenc}
\DeclareSymbolFont{cyrletters}{OT2}{wncyr}{m}{n}
\DeclareMathSymbol{\Sha}{\mathalpha}{cyrletters}{"58}

\usepackage[frenchb]{babel}
\usepackage{amsmath, amssymb, amsthm}

\usepackage[all]{xypic}
\SelectTips{cm}{}

\newcommand{\Z}{\mathbb{Z}}
\newcommand{\Gm}{\mathbb{G}_m}
\newcommand{\Gmks}{\mathbb{G}_{m,k_s}}
\newcommand{\GmA}{\mathbb{G}_{m,A}}
\newcommand{\R}{\mathcal{R}}
\newcommand{\sgr}{\mathfrak{S}}
\newcommand{\Id}{Id}
\newcommand{\qf}[1]{{\langle{#1}\rangle}}

\def\wed{\operatorname{%
  \mathchoice{\extpow}{\extpow}{\wedge}{\wedge}}}

\newcommand{\extpow}{\mbox{\large$\wedge$}}

\DeclareMathOperator{\coker}{coker}
\DeclareMathOperator{\Gal}{Gal}
\DeclareMathOperator{\Hom}{Hom}
\DeclareMathOperator{\Br}{Br}
\DeclareMathOperator{\cor}{cor}
\DeclareMathOperator{\Orth}{\mathbf{O}}
\DeclareMathOperator{\disc}{disc}
\DeclareMathOperator{\Tr}{Tr}
\DeclareMathOperator{\Cl}{Cl}

\newtheorem{lemme}{Lemme}[section]
\newtheorem{prop}[lemme]{Proposition}
\newtheorem{corol}[lemme]{Corollaire}
\newtheorem{thm}[lemme]{Th\'eor\`eme}

\theoremstyle{remark}
\newtheorem{rem}[lemme]{Remarque}
\newtheorem*{remm}{Remarque}

\theoremstyle{definition}
\newtheorem{example}[lemme]{Exemple}

\title{Tores associ\'es \`a une alg\`ebre \'etale quartique}
\author{Jean-Pierre Tignol\thanks{L'auteur remercie Jean-Louis
Colliot-Th\'el\`ene d'avoir attir\'e son attention sur l'article de
Sivatski~\cite{S}, et d'avoir simplifi\'e sa preuve du
th\'eor\`eme~\ref{thm:Sstabrat}. Il b\'en\'eficie de subventions du Fonds de
la Recherche Scientifique--FNRS portant les n$^\circ$~J.0014.15 et J.0149.17.}}
\dedication{\`A Max Knus, en hommage amical pour son
  septante-cinqui\`eme anniversaire}  

\date{14 avril 2017}

\begin{document}

\maketitle

\begin{abstract}
  Cet article met en \'evidence des homomorphismes entre le groupe
  multiplicatif d'une alg\`ebre \'etale de dimension~$4$ et celui d'une
  alg\`ebre \'etale de dimension~$6$ qui lui est canoniquement associ\'ee
  sur un corps arbitraire. Ces homomorphismes sont utilis\'es pour
  mettre en relation diff\'erents groupes de normes et pour donner une
  description de la $2$-torsion dans le groupe de Brauer relatif d'une
  extension de degr\'e~$4$. Les r\'esultats obtenus \'etendent aux alg\`ebres
  \'etales arbitraires de dimension~$4$ plusieurs propri\'et\'es
  remarquables des extensions biquadratiques.
\end{abstract}

\`A toute alg\`ebre \'etale $L$ de dimension~$4$ sur un corps $k$
arbitraire est associ\'ee canoniquement (\`a isomorphisme pr\`es) une
extension quadratique \'etale $S$ d'une $k$-alg\`ebre \'etale cubique
$C$: dans~\cite{KT} (o\`u l'on utilise la notation $\wed_2(L)$ pour
$S$ 
et $\R(L)$ pour $C$), les alg\`ebres $S$ et $C$ sont d\'efinies
comme 
dans~\eqref{eq:defSC} ci-dessous \`a l'aide de la correspondance entre
$k$-alg\`ebres 
\'etales et ensembles finis munis d'une action du groupe de Galois
absolu $\Gamma$ de $k$. En termes de polyn\^omes, cette construction est
tr\`es classique: si $P$ est le polyn\^ome minimal (de degr\'e~$4$) d'un
g\'en\'erateur de $L$, alors $C$ est engendr\'ee par une racine de la
cubique r\'esolvante de $P$, et $S$ est caract\'eris\'ee par la
condition que $P$ se factorise en produit de deux polyn\^omes de
degr\'e~$2$ dans $S[X]$. On peut aussi formuler la d\'efinition de
$S$ et $C$ en termes de cohomologie galoisienne des groupes
sym\'etriques $\mathfrak{S}_n$, comme dans \cite[\S5.2]{KT}. En effet,
les $k$-alg\`ebres \'etales 
quartiques sont classifi\'ees par l'ensemble de cohomologie
$H^1(\Gamma,\mathfrak{S}_4)$, tandis que les extensions quadratiques
\'etales de $k$-alg\`ebres \'etales cubiques sont classifi\'ees par
$H^1(\Gamma,\mathfrak{S}_2\wr\mathfrak{S}_3)$, o\`u
$\mathfrak{S}_2\wr\mathfrak{S}_3$ est le produit couronne de
$\mathfrak{S}_2$ et $\mathfrak{S}_3$. La correspondance entre alg\`ebres
quartiques et extensions quadratiques d'alg\`ebres cubiques apparait
alors comme un avatar de la co\"\i ncidence de diagrammes de Dynkin
$\mathsf{A}_3\equiv\mathsf{D}_3$, qui a pour cons\'equence l'identit\'e
des groupes de Weyl associ\'es
\[
\mathfrak{S}_4=W(\mathsf{A}_3)=W(\mathsf{D}_3) \subset
\mathfrak{S}_2\wr\mathfrak{S}_3.
\]
Nous appellerons $S/C$
\emph{l'extension quadratique r\'esolvante} de $L$. 
\medbreak

Le but de ce travail
est de d\'ecrire les relations entre les 
$k$-tores des groupes multiplicatifs de $L$, $S$ et $C$. Si l'on
note ces tores $T_L$, $T_{S}$ et $T_{C}$,
que l'on d\'esigne par $T^1_{L}$ le noyau de la norme
$T_{L}\to\Gm$ et que l'on d\'efinit semblablement
$T^1_{S}$ et $T^1_{C}$, notre outil principal
consiste en deux morphismes 
\[
\sigma^*\colon T_{L}\to T_{S} \qquad\text{et}\qquad
\tau^*\colon T_{S}\to T_{L}
\]
qui donnent deux suites exactes
\[
1\to T_L/\Gm\to T_S/\Gm\to T_C/\Gm\to 1
\quad\text{et}\quad
1\to T^1_C\to T^1_S\to T^1_L\to 1.
\]
Ces morphismes sont utilis\'es pour obtenir une description de plusieurs
groupes d\'efinis par des normes. En notant $N(L/k)=N_{L/k}(L^\times)$
le groupe des normes de $L/k$, et en d\'efinissant de m\^eme les groupes
de normes $N(S/k)$ et $N(S/C)$, on montre dans la
proposition~\ref{prop:carrenorme}:
\[
\{x\in k^\times\mid x^2\in N(L/k)\} = k^{\times2}\cdot N(S/k)
\quad\text{et}\quad
k^\times\cap N(S/C)=k^{\times2}\cdot N(L/k).
\]
De plus, on donne dans le corollaire~\ref{cor:Brauer} une
param\'etrisation des \'el\'ements de $2$-torsion d\'eploy\'es par $L$ dans le
groupe de Brauer $\Br(k)$: ces \'el\'ements sont des corestrictions
d'alg\`ebres de quaternions sur $C$ d\'eploy\'ees par $S$:
\[
{}_2\Br(L/k)=\{\cor_{C/k}(S/C,x)\mid x\in C^\times\}.
\]
Ces r\'esultats \'etendent aux $k$-alg\`ebres
quartiques arbitraires plusieurs \'enonc\'es qui sont bien connus dans le
cas des extensions biquadratiques, comme le \flqq lemme
biquadratique\frqq\ et le \emph{``Common Slot Theorem''}: voir
l'exemple~\ref{ex:biquadratic}. 
Le th\'eor\`eme~\ref{thm:Albert} montre comment calculer une forme
quadratique d'Albert d'une 
$k$-alg\`ebre de biquaternions d\'eploy\'ee par $L$ \`a partir de sa
description ci-dessus comme corestriction. \`A l'aide de ce th\'eor\`eme,
on retrouve 
dans le corollaire~\ref{cor:quat} la param\'etrisation des 
$k$-alg\`ebres de quaternions d\'eploy\'ees par $L$ donn\'ee dans
\cite[Cor.~22]{HT}, \cite[Th.~6.2]{HS}\footnote{Hoffmann et Sobiech se
  limitent \`a la caract\'eristique~$2$, mais ils envisagent aussi le
  cas o\`u $L$ est une extension ins\'eparable de $k$.} et
\cite[Cor.~4]{S0} (et dans 
\cite[Th.~3.9]{LLT} pour le cas particulier o\`u $L$ est une $2$-extension).
Dans le cas o\`u $k$ est un corps
global, on obtient un principe de Hasse modulo les carr\'es pour les
normes d'extensions quartiques \'etales: voir la
remarque~\ref{rem:Hasse}. Ce principe de Hasse a \'et\'e \'etabli par
Sivatski~\cite{S} par des arguments substantiellement diff\'erents.

Les m\'ethodes utilis\'ees n'imposent aucune restriction sur la
caract\'eristique du corps de base. Sauf mention explicite, $k$ d\'esigne
un corps de base arbitraire.
\bigbreak

\section{$\Gamma$-modules}
\label{sec:ens}

Soit $k_s$ une cl\^oture s\'eparable de $k$ et
$\Gamma=\Gal(k_s/k)$ son groupe de Galois. Par
d\'efinition, un \emph{$\Gamma$-ensemble} est un ensemble fini muni d'une
action continue de $\Gamma$, et un \emph{$\Gamma$-module} est un
$\Z$-module 
de type fini muni d'une action continue de $\Gamma$. Pour tout
$\Gamma$-ensemble $E$, le $\Z$-module libre de base $E$, que l'on
d\'esigne par $\Z[E]$, h\'erite de $E$
une action de $\Gamma$ qui en fait un $\Gamma$-module. Les
modules de ce type sont appel\'es \emph{$\Gamma$-modules de permutation.}
Ils sont \'equip\'es d'homomorphismes \'equivariants 
$\varepsilon_E$ et $\nu_E$ appel\'es respectivement \flqq
augmentation\frqq\ et \flqq norme\frqq
\[
\varepsilon_E\colon \Z[E]\to \Z,\qquad \nu_E\colon \Z\to \Z[E],
\]
d\'efinis sur les \'el\'ements de base par $\varepsilon_E(e)=1$ et
$\nu_E(1)=\sum_{e\in E}e$. On note 
\[
I[E]=\ker\varepsilon_E\qquad\text{et}\qquad J[E]=\coker\nu_E.
\]
On peut identifier tout module de permutation $\Z[E]$ \`a son dual
$\operatorname{Hom}(\Z[E],\Z)$ de sorte que la base $E$ soit
auto-duale. Les homomorphismes $\varepsilon_E$ et $\nu_E$ apparaissent
alors comme duaux l'un de l'autre, et il en est de m\^eme des modules
$I[E]$ et $J[E]$.
\medbreak

Soit \`a pr\'esent $X$ un $\Gamma$-ensemble \`a quatre \'el\'ements. On associe
\`a $X$ le
$\Gamma$-ensemble $\wed^2(X)$
dont les \'el\'ements sont les paires (non ordonn\'ees) d'\'el\'ements de $X$,
et $\R(X)$ le $\Gamma$-ensemble dont les \'el\'ements sont les partitions de
$X$ en sous-ensembles de $2$~\'el\'ements: si $X=\{x_1,x_2,x_3,x_4\}$, alors
\[
\wed^2(X) = \bigl\{
\{x_1,x_2\},\: \{x_1,x_3\},\: \{x_1,x_4\},\:
\{x_2,x_3\},\: \{x_2,x_4\},\: \{x_3,x_4\} \bigr\}
\]
et
\[
\R(X) = \bigl\{ \{\{x_1,x_2\}\{x_3,x_4\}\},\:
\{\{x_1,x_3\}\{x_2,x_4\}\},\:
\{\{x_1,x_4\}\{x_2,x_3\}\} \bigr\}.
\]
L'ensemble $X$ \'etant fix\'e dans la suite, on \'ecrira simplement $\wed^2$
pour $\wed^2(X)$ et $\R$ pour $\R(X)$, afin d'all\'eger les
notations.

Le compl\'ementaire de toute paire d'\'el\'ements de $X$ est aussi une paire
d'\'el\'ements de $X$; on peut donc d\'efinir des applications
$\Gamma$-\'equivariantes comme suit:
\begin{align*}
  \gamma\colon&\wed^2\to\wed^2& \lambda&\mapsto X\setminus\lambda
  & &\text{pour $\lambda\in\wed^2$;}\\
  \pi\colon&\wed^2\to\R&
                               \lambda&\mapsto\{\lambda,\gamma(\lambda)\}
  & &\text{pour $\lambda\in\wed^2$.}
\end{align*}

La projection $\pi\colon\wed^2\to\R$ permet de d\'efinir des
homomorphismes qui sont des analogues
relatifs de l'augmentation et de la norme: on pose
\begin{align*}
  \varepsilon_{{\wed^2}\!/\R}\colon&\Z[\wed^2]\to\Z[\R]&
  \lambda&\mapsto 
  \pi(\lambda) & &\text{pour $\lambda\in\wed^2$;}\\ 
  \nu_{{\wed^2}\!/\R}\colon&\Z[\R]\to\Z[\wed^2]&
  r&\mapsto\sum_{\lambda\in r}\lambda & &\text{pour $r\in\R$}.\\
\intertext{On a aussi un automorphisme \'equivariant}
\Id-\gamma\colon&\Z[\wed^2]\to\Z[\wed^2]& \lambda&\mapsto
  \lambda-\gamma(\lambda) & &\text{pour $\lambda\in\wed^2$.}
\end{align*} 
Choisissons
$\{\lambda_r\mid r\in \R\}$ une section (ensembliste) de $\pi$,
c'est-\`a-dire un ensemble de trois \'el\'ements de $\wed^2$ tels que
$\pi(\lambda_r)=r$ pour tout $r\in \R$. Un \'el\'ement
$\alpha=\sum_{\lambda\in\wed^2}\alpha_\lambda\cdot\lambda\in\Z[\wed^2]$
(avec $\alpha_\lambda\in\Z$ pour $\lambda\in\wed^2$) est dans le noyau
de $\Id-\gamma$ si et seulement si
$\alpha_\lambda=\alpha_{\gamma(\lambda)}$ pour tout
$\lambda\in\wed^2$. Alors $\alpha=\nu_{{\wed^2}\!/\R}\bigl(\sum_{r\in\R}
\alpha_{\lambda_r}\cdot\pi(\lambda_r)\bigr)$. De plus, 
\[
\varepsilon_{{\wed^2}\!/\R}(\alpha) = \sum_{\lambda\in\wed^2}
\alpha_\lambda\cdot\pi(\lambda)=
\sum_{r\in\R}\:\Bigl(\sum_{\pi(\lambda)=r} \alpha_\lambda\Bigr)r,
\]
donc $\varepsilon_{{\wed^2}\!/\R}(\alpha)=0$ si et seulement si
$\alpha_\lambda+\alpha_{\gamma(\lambda)}=0$ pour tout
$\lambda\in\wed^2$. Lorsque cette condition est satisfaite, on a 
\[
\alpha=(\Id-\gamma)\Bigl(\sum_{r\in\R}\alpha_{\lambda_r}\cdot
\lambda_r\Bigr).
\]
Ces observations permettent d'\'etablir l'exactitude de la suite
\begin{equation}
  \label{exseq:1}
  0 \leftarrow \Z[\R] \xleftarrow{\varepsilon_{{\wed^2}\!/\R}}
  \Z[\wed^2] \xleftarrow{\Id-\gamma} \Z[\wed^2]
  \xleftarrow{\nu_{{\wed^2}\!/\R}} \Z[\R] \leftarrow 0.
\end{equation}
D\`es lors, en posant $I[{\wed^2}\!/\R]=\ker\varepsilon_{{\wed^2}\!/\R}$ et
$J[{\wed^2}\!/\R]=\coker\nu_{{\wed^2}\!/\R}$, on voit que l'homomorphisme
$\Id-\gamma$ induit un isomorphisme canonique 
\begin{equation}
 \label{eq:I=J}
I[{\wed^2}\!/\R]\xleftarrow{\sim} J[{\wed^2}\!/\R].
\end{equation}

\subsection{L'homomorphisme $\sigma$}

Avec les notations introduites au d\'ebut de cette section, on consid\`ere
l'homomorphisme \'equivariant
\[
\sigma\colon\Z[\wed^2]\to\Z[X],\qquad
\lambda\mapsto\sum_{x\in\lambda}x \qquad\text{pour
  $\lambda\in\wed^2$.}
\]

\begin{prop}
  \label{prop:1}
  L'homomorphisme $\overline\varepsilon_X\colon \Z[X]\to \Z/2\Z$
  compos\'e de $\varepsilon_X$ et de la r\'eduction modulo~$2$ s'ins\`ere
  dans la suite exacte
  \begin{equation}
  \label{exseq:prop1}
0\leftarrow \Z/2\Z \xleftarrow{\overline\varepsilon_X} \Z[X]
\xleftarrow{\sigma} \Z[\wed^2] \xleftarrow{\nu_{{\wed^2}\!/\R}} I[\R]
\leftarrow 0. 
  \end{equation}
  De plus, les carr\'es suivants sont commutatifs:
  \begin{equation}
  \begin{split}
    \label{diag:1}
    \xymatrix@C=12mm@R=5mm{\Z&\Z[X]\ar[l]_{\varepsilon_X}\\
    \Z\ar[u]^{2}&\Z[\wed^2]\ar[u]_{\sigma}\ar[l]_-{\varepsilon_{\wed^2}}
    }
  \qquad\qquad
  \xymatrix@C=12mm@R=5mm{\Z[X]&\Z\ar[l]_-{\nu_X}\\
  \Z[\wed^2]\ar[u]^{\sigma}&\Z[\R]\ar[l]_-{\nu_{{\wed^2}\!/\R}}
  \ar[u]_{\varepsilon_{\R}}
  }
  \end{split}
  \end{equation}
\end{prop}

\begin{proof}
  La commutativit\'e des diagrammes r\'esulte de calculs directs; la
  nullit\'e de la suite en d\'ecoule. Comme il est clair que
  $\nu_{{\wed^2}\!/\R}$ est injective et $\overline\varepsilon_X$ est
  surjective, il reste seulement \`a \'etablir l'exactitude de la
  suite~\eqref{exseq:prop1} en $\Z[\wed^2]$ et en $\Z[X]$.

  Pour $\alpha=\sum_{\lambda\in\wed^2}\alpha_\lambda\cdot\lambda\in
  \Z[\wed^2]$ (avec $\alpha_\lambda\in\Z$ pour tout
  $\lambda\in\wed^2$), on a
  \[
  \sigma(\alpha)=\sum_{\lambda\in\wed^2}\alpha_\lambda
  \Bigl(\sum_{x\in\lambda}x\Bigr) = \sum_{x\in
    X}\:\Bigl(\sum_{\lambda\ni x}\alpha_\lambda\Bigr)x,
  \]
  donc $\sigma(\alpha)=0$ si et seulement si $\sum_{\lambda\ni
    x}\alpha_\lambda=0$ pour tout $x\in X$. Notons
  $X=\{x_1,\,x_2,\,x_3,\,x_4\}$. Comme
  \[
  2\alpha_{\{x_1,x_2\}}-2\alpha_{\{x_3,x_4\}} = \Bigl(\sum_{\lambda\ni
    x_1}\alpha_\lambda\Bigr) + \Bigl(\sum_{\lambda\ni
    x_2}\alpha_\lambda\Bigr) - \Bigl(\sum_{\lambda\ni
    x_3}\alpha_\lambda\Bigr) - \Bigl(\sum_{\lambda\ni
    x_4}\alpha_\lambda\Bigr),
  \]
  on voit que $\alpha_{\{x_1,x_2\}}=\alpha_{\{x_3,x_4\}}$ pour
  $\alpha\in\ker\sigma$. De m\^eme,
  $\alpha_{\{x_1,x_3\}}=\alpha_{\{x_2,x_4\}}$ et $\alpha_{\{x_1,x_4\}}
  = \alpha_{\{x_2,x_3\}}$ pour $\alpha\in\ker\sigma$, donc
  $\alpha=\gamma(\alpha)$ si $\sigma(\alpha)=0$. Vu l'exactitude de la
  suite \eqref{exseq:1}, on peut donc trouver pour tout
  $\alpha\in\ker\sigma$ un \'el\'ement $\beta\in\Z[\R]$ tel que $\alpha=
  \nu_{{\wed^2}\!/\R}(\beta)$. La commutativit\'e du diagramme de droite
  de \eqref{diag:1} montre que $\varepsilon_\R(\beta)=0$, donc
  $\beta\in I[\R]$. La suite~\eqref{exseq:prop1} est donc exacte en
  $\Z[\wed^2]$.

  Pour \'etablir l'exactitude en $\Z[X]$, consid\'erons $\xi=\sum_{x\in
    X}\xi_x\cdot x\in \Z[X]$ (avec $\xi_x\in\Z$ pour tout $x\in X$)
  tel que $\sum_{x\in X}\xi_x=2\zeta$ pour un certain
  $\zeta\in\Z$. Pour tout $\lambda\in\wed^2$, posons
  \[
  \alpha_\lambda= \Bigl(\sum_{x\in\lambda}\xi_x\Bigr)-\zeta=
  \frac12\Bigl(\sum_{x\in\lambda} \xi_x -
  \sum_{y\in\gamma(\lambda)}\xi_y\Bigr) \in\Z.
  \]
  Alors $\alpha_{\gamma(\lambda)}=-\alpha_\lambda$ pour tout
  $\lambda\in\wed^2$, donc l'\'el\'ement
  $\alpha=\sum_{\lambda\in\wed^2}\alpha_\lambda\cdot\lambda$ est dans
  le noyau de $\varepsilon_{{\wed^2}\!/\R}$. Vu l'exactitude de la
  suite \eqref{exseq:1}, on peut trouver
  $\beta=\sum_{\lambda\in\wed^2}\beta_\lambda\cdot\lambda\in\Z[\wed^2]$
  tel que $\alpha=\beta-\gamma(\beta)$, ce qui signifie que
  \[
  \beta_\lambda-\beta_{\gamma(\lambda)} = {\frac12}
  \Bigl(\sum_{x\in\lambda}\xi_x -
  \sum_{y\in\gamma(\lambda)}\xi_y\Bigr)
  \quad\text{pour tout $\lambda\in\wed^2$.}
  \]
  Pour un \'el\'ement $z\in X$ fix\'e on a alors d'une part
  \[
  \sum_{\lambda\ni z}(\beta_\lambda-\beta_{\gamma(\lambda)}) =
  \sum_{\lambda\ni z}\beta_\lambda - \sum_{\lambda\not\ni
    z}\beta_\lambda = 2\Bigl(\sum_{\lambda\ni z}\beta_\lambda\Bigr)
  -\varepsilon_{\wed^2}(\beta),
  \]
  et d'autre part
  \[
  \sum_{\lambda\ni z}(\beta_\lambda-\beta_{\gamma(\lambda)}) =
  {\frac12}\sum_{\lambda\ni z}\:
  \Bigl(\sum_{x\in\lambda}\xi_x -
  \sum_{y\in\gamma(\lambda)}\xi_y\Bigr) = {\frac12}
  \bigl(3\xi_z - \sum_{x\neq z}\xi_x\bigr) = 2\xi_z-\zeta.
  \]
  En comparant ces \'equations, on obtient
  \begin{equation}
    \label{eq:compa}
    2\Bigl(\sum_{\lambda\ni z}\beta_\lambda\Bigr) -
    \varepsilon_{\wed^2}(\beta) = 2\xi_z-\zeta,
  \end{equation}
  ce qui montre que
  $\varepsilon_{\wed^2}(\beta)\equiv\zeta\bmod2$. Or, la condition
  $\alpha=\beta-\gamma(\beta)$ ne d\'etermine $\beta$ qu'\`a un \'el\'ement de
  l'image de $\nu_{{\wed^2}\!/\R}$ pr\`es, et
  $\varepsilon_{\wed^2}(\operatorname{im}\nu_{{\wed^2}\!/\R})=2\Z$,
  donc quitte \`a ajuster $\beta$ on peut supposer
  $\varepsilon_{\wed^2}(\beta)=\zeta$. L'\'equation~\eqref{eq:compa},
  qui vaut pour tout $z\in X$, donne alors $\sum_{\lambda\ni
    z}\beta_\lambda = \xi_z$, donc $\xi=\sigma(\beta)$.
\end{proof}

\begin{corol}
  \label{corol:1}
  La restriction de $\sigma$ \`a $I[\wed^2]$ est un homomorphisme
  $\sigma_I\colon I[\wed^2]\to I[X]$ qui s'ins\`ere dans la suite
  exacte suivante:
  \begin{equation}
  \label{exseq:corol2}
  0 \leftarrow I[X] \xleftarrow{\sigma_I} I[\wed^2]
  \xleftarrow{\nu_{{\wed^2}\!/\R}} I[\R] \leftarrow 0.
  \end{equation}
\end{corol}

\begin{proof}
  Le diagramme de gauche de \eqref{diag:1} s'\'etend en un diagramme
  commutatif dont les lignes sont exactes:
  \[
  \xymatrix@C=12mm@R=5mm{0&\Z\ar[l]& \Z[X]\ar[l]_{\varepsilon_X}& I[X]\ar[l]&
    0\ar[l]\\
  0&\Z\ar[l]\ar[u]^{2}&\Z[\wed^2]\ar[l]_{\varepsilon_{\wed^2}}
  \ar[u]^{\sigma} &
  I[\wed^2]\ar[l] \ar[u]_{\sigma_I}&0\ar[l]
  }
  \]
  Comme $\coker(2)=\coker\sigma$ et $\ker(2)=0$, le corollaire
  s'obtient en appliquant le lemme du serpent \`a ce diagramme.
\end{proof}

On peut raisonner de m\^eme avec le diagramme de droite de
\eqref{diag:1}. On obtient une suite exacte
\[
0\leftarrow \Z/2\Z \leftarrow J[X] \leftarrow J[\wed^2/\R] \leftarrow 0
\]
avec des homomorphismes induits par $\overline\varepsilon_X$ et par
$\sigma$. Il est cependant utile d'\'eliminer les \'el\'ements de
torsion en modifiant quelque peu $\sigma$ et $\nu_X$: d'abord, on peut
remplacer la suite exacte~\eqref{exseq:prop1} par la suivante:
\begin{equation}
\label{exseq:prop1bis}
0 \leftarrow \Z \xleftarrow{(\varepsilon_X,-2)} \Z[X]\times\Z
\xleftarrow{\sigma\times\varepsilon_{\wed^2}} \Z[\wed^2]
\xleftarrow{\nu_{{\wed^2}\!/\R}} I[\R] \leftarrow0.
\end{equation}
Ensuite, on peut remplacer le carr\'e de droite de \eqref{diag:1} par le suivant:
\begin{equation}
  \label{diag:2}
  \begin{split}
  \xymatrix@C=12mm@R=5mm{
  \Z[X]\times\Z&\Z\ar[l]_-{(\nu,2)}\\
  \Z[\wed^2]\ar[u]^{\sigma\times\varepsilon_{\wed^2}} &
  \Z[\R]\ar[l]_-{\nu_{{\wed^2}\!/\R}} \ar[u]_{\varepsilon_\R}}
  \end{split}
\end{equation}

\begin{corol}
  \label{corol:2}
  Soit $J'[X]=\coker(\nu,2)$. On a une suite exacte
  \begin{equation}
  \label{exseq:corol3}
  0\leftarrow \Z\xleftarrow{\varepsilon'}J'[X] \xleftarrow{\sigma'}
  J[{\wed^2}\!/\R] \leftarrow 0
  \end{equation}
  o\`u les homomorphismes $\varepsilon'$ et $\sigma'$ sont induits par
  $(\varepsilon_X,-2)$ et $\sigma\times\varepsilon_{\wed^2}$
  respectivement. 
\end{corol}

\begin{proof}
  On applique le lemme du serpent au diagramme commutatif suivant,
  obtenu en \'etendant le diagramme~\eqref{diag:2}:
  \begin{equation}
  \label{diag:corol3}
  \begin{split}
  \xymatrix@R=5mm{
  0&J'[X]\ar[l]&\Z[X]\times\Z\ar[l] &\Z\ar[l]_-{(\nu,2)}&0\ar[l]\\
  0&J[{\wed^2}\!/\R]\ar[l]\ar[u]^{\sigma'} & \Z[\wed^2]\ar[l]
  \ar[u]^{\sigma\times\varepsilon_{\wed^2}} & \Z[\R]\ar[l]
    \ar[u]_{\varepsilon_\R} &0\ar[l]
  }
  \end{split}
  \end{equation}
\end{proof}

\subsection{L'homomorphisme $\tau$}

Rappelons que les $\Gamma$-modules de permutation sont auto-duaux. On
peut donc d\'efinir un homomorphisme $\tau\colon\Z[X]\to\Z[\wed^2]$ en
prenant le dual de $\sigma$. De mani\`ere explicite,
\[
\tau(x)=\sum_{\lambda\ni x}\lambda\qquad\text{pour $x\in X$.}
\]
Par dualit\'e \`a partir de~\eqref{diag:1} (ou par un calcul direct) on
obtient les carr\'es commutatifs suivants:
\begin{equation}
  \begin{split}
    \label{diag:3}
    \xymatrix@C=12mm@R=5mm{\Z[\wed^2]&\Z\ar[l]_-{\nu_{\wed^2}}\\
    \Z[X]\ar[u]^{\tau}&\Z\ar[u]_{2}\ar[l]_-{\nu_{X}}
    }
  \qquad\qquad
  \xymatrix@C=12mm@R=5mm{\Z[\R]&\Z[\wed^2]
    \ar[l]_-{\varepsilon_{{\wed^2}\!/\R}}\\  
  \Z\ar[u]^{\nu_\R}&\Z[X]\ar[l]_-{\varepsilon_X}
  \ar[u]_{\varepsilon_{\R}}
  }
  \end{split}
\end{equation}
On a aussi une suite exacte duale de \eqref{exseq:prop1bis}:
\begin{equation}
\label{exseq:prop1bisdual}
0\leftarrow J[\R] \xleftarrow{\varepsilon_{{\wed^2}\!/\R}} \Z[\wed^2]
\xleftarrow{(\tau,\nu_{\wed^2})} \Z[X]\times\Z
\xleftarrow{(\nu_X\times-2)} \Z\leftarrow0.
\end{equation}

\begin{prop}
  \label{prop:2}
  L'homomorphisme $\tau$ induit un homomorphisme $\tau_J\colon J[X]\to
  J[\wed^2]$ et la suite suivante est exacte:
  \begin{equation}
    \label{exseq:prop2}
  0\leftarrow J[\R] \xleftarrow{\varepsilon_{{\wed^2}\!/\R}} J[\wed^2]
  \xleftarrow{\tau_J} J[X] \leftarrow0.
  \end{equation}
\end{prop}

\begin{proof}
  Le noyau et le conoyau de $(\tau,\nu_{\wed^2})$ \'etant donn\'es par la
  suite~\eqref{exseq:prop1bisdual}, il suffit d'appliquer le lemme du
  serpent au diagramme suivant, dont les lignes sont exactes:
  \[
  \xymatrix@C=12mm@R=5mm{
  0&J[\wed^2]\ar[l]&\Z[\wed^2]\ar[l]&\Z\ar[l]_-{\nu_{\wed^2}}&0\ar[l]\\
  0&J[X]\ar[l]\ar[u]^{\tau_J}&\Z[X]\times\Z\ar[l]_-{(\text{can},0)}
  \ar[u]_{(\tau,\nu_{\wed^2})}&
  \Z\times\Z\ar[l]_-{\nu_X\times\Id_\Z}\ar[u]_{(2,\Id_\Z)}&0\ar[l]
  }
  \]
\end{proof}

\section{Tores}
\label{sec:tores}

Comme dans la section pr\'ec\'edente, on d\'esigne par $k$ un corps
arbitraire; on fixe une cl\^oture s\'eparable $k_s$ de $k$ et on pose
$\Gamma=\Gal(k_s/k)$. En associant \`a toute $k$-alg\`ebre \'etale $A$
l'ensemble $E(A)=\Hom_{\text{$k$-alg}}(A,k_s)$ des homomorphismes de
$k$-alg\`ebres de $A$ dans $k_s$, on d\'efinit une anti-\'equivalence de
cat\'egories entre les $k$-alg\`ebres \'etales et les $\Gamma$-ensembles. De
m\^eme, on a une anti-\'equivalence de cat\'egories entre $k$-tores et
$\Gamma$-modules, qui associe \`a tout $k$-tore $T$ le module de ses
caract\`eres $\widehat T=\Hom(T_s,\Gmks)$. Ces deux anti-\'equivalences
sont li\'ees de la mani\`ere suivante: toute $k$-alg\`ebre \'etale $A$ d\'efinit
d'une part un $\Gamma$-ensemble $E(A)$ et d'autre part un $k$-tore
\[
T_A=R_{A/k}(\GmA)
\]
obtenu \`a partir du groupe multiplicatif $\GmA$ de $A$ par restriction
des scalaires (\`a la Weil) de $A$ \`a $k$. Ces deux constructions sont
li\'ees par
\[
\widehat T_A=\Z[E(A)].
\]
La correspondance fonctorielle entre $\Gamma$-modules et tores associe
\`a l'augmentation $\varepsilon_{E(A)}$ l'inclusion $i_A\colon\Gm\to
T_A$, et \`a 
la norme $\nu_{E(A)}$ la norme $N_A\colon T_A\to\Gm$. D\`es lors, le
tore correspondant au module $I[E(A)]$ est le quotient $T_A/\Gm$, et
$J[E(A)]$ correspond au tore $R^1_{A/k}(\GmA)$ noyau de la norme
$N_A$, que l'on notera simplement $T^1_A$.
\medbreak

\`A partir d'ici, fixons une $k$-alg\`ebre \'etale $L$ de dimension~$4$,
et notons simplement $X$ pour le $\Gamma$-ensemble $E(L)$, qui a
quatre \'el\'ements. L'extension quadratique r\'esolvante $S/C$ de $L$
est d\'efinie \`a isomorphisme pr\`es par
les conditions
\begin{equation}
\label{eq:defSC}
E(S)=\wed^2(X)\qquad\text{et}\qquad E(C)=\R(X).
\end{equation}
(L'application \'equivariante $\pi\colon\wed^2(X)\to\R(X)$ du
\S\ref{sec:ens} d\'efinit un homomorphisme injectif d'alg\`ebres
$C\to S$ que l'on consid\`ere comme une inclusion.) Les
homomorphismes $\sigma$ et $\tau$ du \S\ref{sec:ens} donnent des
homomorphismes de $k$-tores
\[
\sigma^*\colon T_L\to T_S \qquad\text{et}\qquad \tau^*\colon
T_S\to T_L,
\]
et l'on peut interpr\'eter en termes de tores les suites exactes et les
diagrammes commutatifs du \S\ref{sec:ens}. On trouve ainsi les
diagrammes commutatifs suivants d\'eduits de \eqref{diag:1}
\begin{equation}
  \begin{split}
    \label{diag:4}
    \xymatrix@C=12mm@R=5mm{\Gm\ar[r]^{i_L}\ar[d]_{2}&T_L\ar[d]^{\sigma^*}\\
    \Gm\ar[r]^{i_S}&T_S
    }
  \qquad\qquad
  \xymatrix@C=12mm@R=5mm{T_L\ar[r]^{N_L}\ar[d]_{\sigma^*}&\Gm\ar[d]^{i_C}\\
  T_S\ar[r]^{N_{S/C}}& T_C
  }
  \end{split}
\end{equation}
et ceux d\'eduits de \eqref{diag:3}
  \begin{equation}
  \begin{split}
    \label{diag:5}
    \xymatrix@C=12mm@R=5mm{T_S\ar[r]^{N_S}\ar[d]_{\tau^*}&\Gm\ar[d]^{2}\\
    T_L\ar[r]^{N_L}&\Gm
    }
  \qquad\qquad
  \xymatrix@C=12mm@R=5mm{T_C\ar[r]^{i_{S/C}}\ar[d]_{N_C}&
    T_S\ar[d]^{\tau^*}\\ 
  \Gm\ar[r]^{i_L}&T_L
  }
  \end{split}
\end{equation}

On obtient aussi \`a partir de \eqref{exseq:corol2} et
\eqref{exseq:prop2} les suites exactes
\[
1\to T_L/\Gm \xrightarrow{\sigma_I^*} T_S/\Gm
\xrightarrow{N_{S/C}} T_C/\Gm\to1
\]
et
\begin{equation}
  \label{exseq:2}
  1\to T^1_C \xrightarrow{i_{S/C}} T^1_S
  \xrightarrow{\tau_J^*} T^1_L\to1.
\end{equation}

\medbreak

Consid\'erons \`a pr\'esent le tore $U$ noyau de l'homomorphisme
\[
(N_L,2)\colon T_L\times\Gm\to\Gm.
\] 
Son $\Gamma$-module des caract\`eres
est le module $J'[X]$ introduit dans le corollaire~\ref{corol:2}. En
notant $T^1_{S/C}$ le noyau de $N_{S/C}\colon T_S\to
T_C$ (dont le module des caract\`eres est $J[{\wed^2}\!/\R]$) on obtient
\`a partir de \eqref{exseq:corol3} la suite exacte
\begin{equation}
  \label{exseq:S}
  1\to \Gm\to U\xrightarrow{{\sigma'}^*} T^1_{S/C}\to1.
\end{equation}

\begin{thm}
  \label{thm:Sstabrat}
  Le $k$-tore $U$ est stablement rationnel.
\end{thm}

\begin{proof}
  Montrons d'abord que le $k$-tore $T^1_{S/C}$ est stablement
  rationnel. 
  L'isomorphisme $I[{\wed^2}\!/\R]\simeq J[{\wed^2}\!/\R]$
  de~\eqref{eq:I=J} entraine un isomorphisme de tores
  $T_S/T_C\simeq T^1_{S/C}$ induit par $\Id-\gamma$. On a
  donc une suite exacte
  \[
  1\to T_C\to T_S \xrightarrow{\Id-\gamma} T^1_{S/C}\to 1.
  \]
  La fibre de $\Id-\gamma$ au point g\'en\'erique de $T^1_{S/C}$ est
  un torseur sous $T_C$. Or le lemme de Shapiro et le th\'eor\`eme~90 de
  Hilbert donnent $H^1(k(T^1_{S/C}),T_C)=1$, donc $T_S$ est
  birationnellement \'equivalent \`a $T^1_{S/C}\times T_C$. Comme
  $T_S$ et $T_C$ sont $k$-rationnels, on voit que
  $T^1_{S/C}$ est stablement rationnel.

  Consid\'erons ensuite la suite exacte~\eqref{exseq:S}. Comme
  $H^1(k(T^1_{S/C}),\Gm)=1$, le m\^eme argument que dans la
  premi\`ere partie de la preuve
  montre que $U$ est birationnellement \'equivalent \`a
  $T^1_{S/C}\times\Gm$, donc $U$ est stablement rationnel.
\end{proof}

\section{Normes et groupe de Brauer relatif}
\label{sec:normes}

Conservons les notations du \S\ref{sec:tores}, et consid\'erons \`a
pr\'esent la cohomologie (\'etale) des tores qui y sont d\'efinis. Pour
toute $k$-alg\`ebre \'etale $A$, on note $N_{A/k}\colon A^\times\to
k^\times$ la norme, et
\[
A^1=\ker(N_{A/k})\subset A^\times,\qquad
N(A/k)=N_{A/k}(A^\times)\subset k^\times.
\]
Le th\'eor\`eme~90 de Hilbert et le lemme de Shapiro entrainent
$H^1(k,T_A)=1$; d\`es lors de la suite exacte qui d\'efinit $T_A^1$
\[
1\to T^1_A\to T_A\xrightarrow{N_A}\Gm\to1,
\]
on d\'eduit que $H^1(k,T_A^1)=A^\times/N(A/k)$.

\begin{prop}
  \label{prop:carrenorme}
  Soit $S/C$ l'extension quadratique r\'esolvante de la
  $k$-alg\`ebre \'etale quartique $L$. On a
  \[
  \{x\in k^\times\mid x^2\in N(L/k)\} = k^{\times2}\cdot N(S/k)
\quad\text{et}\quad
  k^\times\cap N(S/C)=k^{\times2}\cdot N(L/k).
  \]
\end{prop}

\begin{proof}
  On a un diagramme commutatif dont les lignes sont exactes
  \[
  \xymatrix@R=3mm{
  1\ar[r]&T^1_C\ar[r]\ar[d]& T_C\ar[r]^{N_C}\ar[d]_{i_{S/C}}
  &\Gm\ar[r]\ar[d]^{2} &1\\
  1\ar[r]&T^1_S\ar[r]\ar[d]&
  T_S\ar[r]^{N_S}\ar[d]_{\tau^*} 
  &\Gm\ar[r]\ar[d]^{2} &1\\
  1\ar[r]&T^1_L\ar[r]& T_L\ar[r]^{N_L}
  &\Gm\ar[r] &1
  }
  \]
  La premi\`ere colonne de ce diagramme est exacte par la
  proposition~\ref{prop:2} (voir \eqref{exseq:2}). On en d\'eduit en
  cohomologie un diagramme commutatif dont la deuxi\`eme colonne est
  exacte
  \[
  \xymatrix@R=3mm{k^\times\ar[r]\ar[d]_{2}&k^\times/N(C/k)\ar[d]\\
  k^\times\ar[r]\ar[d]_{2}&k^\times/N(S/k)\ar[d]\\
  k^\times\ar[r]&k^\times/N(L/k)
  }
  \]
  La premi\`ere \'egalit\'e en d\'ecoule. Pour \'etablir la deuxi\`eme, on
  consid\`ere le diagramme commutatif d\'eduit de \eqref{diag:corol3}; ses
  lignes sont exactes:
  \begin{equation}
  \label{diag:carrenorme}
  \begin{split}
  \xymatrix@C=13mm@R=5mm{1\ar[r]&U\ar[r]\ar[d]_{{\sigma'}^*}&
    T_L\times\Gm\ar[r]^{(N_L,2)}\ar[d]^{(\sigma,i_S)}&
      \Gm\ar[r]\ar[d]^{i_C}&1\\ 
  1\ar[r]&T^1_{S/C}\ar[r]&T_S\ar[r]^{N_{S/C}}&
  T_C\ar[r]&1
  }
  \end{split}
  \end{equation}
  On en d\'eduit la commutativit\'e du carr\'e suivant, o\`u $\sigma^1$ est
  l'application induite par ${\sigma'}^*$:
  \[
  \xymatrix@R=5mm{
  k^\times\ar[r]\ar[d]_{i_C}&H^1(k,U)\ar[d]^{\sigma^1}\\
  C^\times\ar[r]&H^1(k,T^1_{S/C})
  }
  \]
  Par ailleurs, l'exactitude des lignes du
  diagramme~\eqref{diag:carrenorme} permet d'identifier
  \begin{equation}
  \label{eq:H1}
  H^1(k,U)=k^\times/N(L/k)k^{\times2} \qquad\text{et}\qquad
  H^1(k,T^1_{S/C})=C^\times/N(S/C),
  \end{equation}
  donc l'application $\sigma^1\colon k^\times/N(L/k)k^{\times2}\to
  C^\times/N(S/C)$ est induite par l'inclusion $k\subset C$. Or,
  le corollaire~\ref{corol:2} donne la suite exacte
  \[
  1\to\Gm\to U\xrightarrow{{\sigma'}^*} T^1_{S/C}\to 1,
  \]
  qui montre que $\sigma^1$ est injective. La deuxi\`eme \'egalit\'e s'en
  d\'eduit. 
\end{proof}

\begin{rem}
  \label{rem:Hasse}
  Comme le tore $U$ est stablement rationnel, le principe de Hasse
  vaut pour $H^1(k,U)$ lorsque $k$ est un corps global: voir
  \cite[\S8]{CTS}. Vu l'interpr\'etation de $H^1(k,U)$ en \eqref{eq:H1},
  on a donc un principe de Hasse modulo les carr\'es pour les normes
  d'extensions quartiques \'etales. On retrouve ainsi (en caract\'eristique
  arbitraire) le th\'eor\`eme~2.1 de \cite{S}.
\end{rem}

Consid\'erons \`a pr\'esent les groupes de cohomologie \'etale (ou
galoisienne) de degr\'e~2. Le 
groupe $H^2(k,\Gm)$ s'identifie au groupe de Brauer $\Br(k)$. Plus
g\'en\'eralement,
\[
H^2(k,T_A)=\Br(A)
\]
pour toute $k$-alg\`ebre \'etale $A$, et la norme $T_A\to\Gm$ induit la
corestriction
\[
\cor_{A/k}\colon\Br(A)\to\Br(k).
\]
Si $A\subset B$ est une extension d'alg\`ebres \'etales, on note
$\Br(B/A)$ le noyau de l'homomorphisme d'extension des scalaires
$\Br(A)\to\Br(B)$, et $_2\Br(B/A)$ le sous-groupe de $2$-torsion de
$\Br(B/A)$.

\begin{thm}
  \label{thm:Brauer}
  Soit $L$ une $k$-alg\`ebre \'etale quartique, et soit $S/C$
  l'extension r\'esol\-vante associ\'ee. On a un diagramme commutatif dont
  les lignes sont exactes:
  \begin{equation}
  \label{diag:Brauer}
  \begin{split}
  \xymatrix{
  1\ar[r]&C^\times\ar[d]_{N_{C/k}}\ar[r]&
  S^\times\ar[d]^{\tau^*\times
    -N_{S/k}}\ar[r]^{\Id-\gamma}&S^\times \ar@{=}[d]
  \ar[r]^{N_{S/C}}& 
  C^\times\ar[d]\ar[r]^-{\delta}& \Br(S/C)\ar[d]^{\cor_{C/k}}
  \ar[r] &0\\
  1\ar[r]&k^\times\ar[r]^-{i_L\times-2}& L^\times\times k^\times
  \ar[r]^-{\sigma^*,i_S}& S^\times\ar[r]^-{N_{S/C}}&
    C^\times/k^\times \ar[r]^-{\overline \delta}&{}_2\Br(L/k)\ar[r]&0
  }
  \end{split}
  \end{equation}
\end{thm}

\begin{proof}
  Soit $X$ le $\Gamma$-ensemble \`a quatre \'el\'ements $E(L)$ canoniquement
  associ\'e \`a $L$. Comme dans le \S\ref{sec:ens}, on consid\`ere
  $\wed^2=\wed^2(X)$ et $\R=\R(X)$. On a un diagramme commutatif de
  $\Gamma$-modules, dont les lignes sont exactes (la ligne inf\'erieure
  est la suite exacte \eqref{exseq:prop1bis}):
  \[
  \xymatrix{0&\Z[\R]\ar[l]&
    \Z[\wed^2]\ar[l]_{\varepsilon_{{\wed^2}\!/\R}}&
    \Z[\wed^2]\ar[l]_{\Id-\gamma}& \Z[\R]\ar[l]_{\nu_{{\wed^2}\!/\R}}&
    0\ar[l] \\
  0&\Z\ar[l]\ar[u]^{\nu_\R}& \Z[X]\times\Z\ar[l]_-{(\varepsilon_X,-2)}
  \ar[u]_{\tau,-\nu_{\wed^2}}&
  \Z[\wed^2]\ar[l]_-{\sigma\times\varepsilon_{\wed^2}} \ar@{=}[u]&
  I[\R]\ar[l]\ar[u]& 0\ar[l]
  }
  \]
  On en d\'eduit un diagramme commutatif de tores, dont les lignes sont
  exactes:
  \begin{equation}
  \label{diag:Brauer1}
  \begin{split}
  \xymatrix{1\ar[r]&T_C\ar[d]_{N_C}\ar[r]&
    T_S\ar[d]^{\tau^*\times -N_S}\ar[r]^{\Id-\gamma}&
    T_S\ar@{=}[d]\ar[r]^{N_{S/C}}& T_C\ar[r]\ar[d]&1\\
  1\ar[r]&\Gm\ar[r]^-{i_L\times-2}&T_L\times\Gm
  \ar[r]^-{(\sigma^*,i_S)}& T_S\ar[r]&
  T_C/\Gm\ar[r]&1
  }
  \end{split}
  \end{equation}
  Soient
  \[
  T^1_{S/C}=\ker(N_{S/C}\colon T_S\to T_C)
  \quad\text{et}\quad T=\ker(N_{S/C}\colon T_S\to
  T_C/\Gm).
  \]
  On a un morphisme canonique (d'inclusion) $T^1_{S/C}\to
  T$, et le diagramme~\eqref{diag:Brauer1} se partage en deux:
  \begin{equation}
  \label{diag:Brauer2}
  \parbox{3cm}{
  \xymatrix@C=3mm@R=3mm{
  1\ar[r]&T_C\ar[r]\ar[d]&T_S\ar[d]\ar[r]&
  T^1_{S/C}\ar[d]\ar[r]&1\\
  1\ar[r]&\Gm\ar[r]&T_L\times\Gm\ar[r]&T\ar[r]&1
  }
  }
  \quad
  \text{et}
  \quad
  \parbox{3cm}{
  \xymatrix@C=3mm@R=3mm{
  1\ar[r]&T^1_{S/C}\ar[r]\ar[d]& T_S\ar@{=}[d]\ar[r]&
  T_C\ar[d]\ar[r]&1\\
  1\ar[r]&T\ar[r]&T_S\ar[r]& T_C/\Gm\ar[r]&1
  }
  }
  \end{equation}
  Le diagramme de gauche donne en cohomologie les diagrammes
  commutatifs suivants, dont les lignes sont exactes:
  \begin{equation}
  \label{diag:Brauer3}
  \begin{split}
  \xymatrix@R=5mm@C=5mm{
  1\ar[r]&C^\times\ar[r]\ar[d]_{N_{C/k}}&S^\times\ar[r]^-{\Id-\gamma}
  \ar[d]^{\tau^*\times -N_{S/k}}& T^1_{S/C}(k)\ar[d]\ar[r]&1\\
  1\ar[r]&k^\times\ar[r]&L^\times\times k^\times\ar[r]&T(k)\ar[r]&1}
  \end{split}
  \end{equation}
  et
  \[
  \xymatrix@C=10mm@R=5mm{1\ar[r]&H^1(k,T^1_{S/C})\ar[d]\ar[r]&
    \Br(C)\ar[d]_{\cor_{C/k}}\ar[r]& \Br(S)\ar[d]\\
  1\ar[r]& H^1(k,T)\ar[r]&\Br(k)\ar[r]^-{i_L\times-2}&
  \Br(L)\times\Br(k)
  }
  \]
  De l'exactitude des suites de ce dernier diagramme, on d\'eduit des
  isomorphismes canoniques
  \begin{equation}
    \label{eq:Brauer}
    H^1(k,T^1_{S/C})=\Br(S/C) \qquad\text{et}\qquad
    H^1(k,T)={}_2\Br(L/k).
  \end{equation}
  Par ailleurs, le diagramme de droite de~\eqref{diag:Brauer2} donne
  en cohomologie le diagramme commutatif suivant, dont les lignes sont
  exactes:
  \begin{equation}
    \label{diag:Brauer4}
    \begin{split}
  \xymatrix@R=5mm{1\ar[r]& T^1_{S/C}(k)\ar[d]\ar[r]&
    S^\times\ar@{=}[d] \ar[r]& C^\times\ar[d]\ar[r]&
    H^1(k,T^1_{S/C})\ar[d]\ar[r]& 1\\
  1\ar[r]& T(k)\ar[r]& S^\times\ar[r]& C^\times/k^\times\ar[r]&
  H^1(k,T)\ar[r]&1} 
    \end{split}
  \end{equation}
  Le th\'eor\`eme se d\'eduit en recollant les
  diagrammes~\eqref{diag:Brauer3} et \eqref{diag:Brauer4}, et en
  utilisant~\eqref{eq:Brauer}.
\end{proof}

\begin{rem}
  \label{rem:precis}
  L'exactitude de la ligne inf\'erieure de~\eqref{diag:Brauer} en
  $S^\times$ donne une autre preuve de l'\'egalit\'e $k^\times\cap
  N(S/C)=k^{\times2}\cdot N(L/k)$ de la
  proposition~\ref{prop:carrenorme}, avec une pr\'ecision
  suppl\'ementaire: si $y\in S^\times$ est tel que
  $N_{S/C}(y)\in k^\times$, alors il existe $x\in
  L^\times$ et $\lambda\in k^\times$ tels que $y=\lambda\sigma^*(x)$ (ce
  qui entraine $N_{S/C}(y)=\lambda^2N_{L/k}(x)$ vu~\eqref{diag:4}).
\end{rem}

Pour d\'ecrire les homomorphismes $\delta$ et $\overline\delta$ qui
apparaissent dans le diagramme~\eqref{diag:Brauer}, on introduit la
notation suivante: pour une extension quadratique \'etale $A\subset B$
de $k$-alg\`ebres \'etales, dont l'automorphisme non trivial est not\'e
$\rho$, et pour $a\in A^\times$, on consid\`ere la $A$-alg\`ebre de
quaternions
\[
(B/A,a)=B\oplus Bz
\]
o\`u la multiplication est d\'efinie par
\[
z^2=a\qquad\text{et}\qquad zb=\rho(b)z\quad\text{pour $b\in B$.}
\]
L'alg\`ebre $(B/A,a)$ est une alg\`ebre d'Azumaya sur $A$; on d\'esigne
encore par $(B/A,a)$ sa classe de Brauer dans $\Br(A)$. Le calcul
montre que l'application $\delta$ est donn\'ee par
\[
\delta(x)=(S/C,x)\qquad\text{pour $x\in C^\times$}.
\]
D\`es lors, vu la commutativit\'e du diagramme~\eqref{diag:Brauer}, on a
\[
\overline\delta(xk^\times)=(S/C,x)\qquad\text{pour $x\in
  C^\times$}.
\]

\begin{corol}
  \label{cor:Brauer}
  Avec les notations du th\'eor\`eme~\ref{thm:Brauer}, 
  \[
  {}_2\Br(L/k)= \{\cor_{C/k}(S/C,x)\mid x\in C^\times\}.
  \]
  De plus, pour $x\in C^\times$ la classe de
  Brauer $\cor_{C/k}(S/C,x)$ est d\'eploy\'ee si et seulement si
  $x\in k^\times\cdot N(S/C)$.
\end{corol}

\begin{proof}
  Vu la d\'efinition de $\overline\delta$, le corollaire d\'ecoule
  imm\'ediatement de l'exactitude de la ligne inf\'erieure du
  diagramme~\eqref{diag:Brauer} en $_2\Br(L/k)$ et en $C^\times/k^\times$.
\end{proof}

\section{Formes d'Albert}
\label{sec:Albert}

\`A toute alg\`ebre simple centrale $A$ de degr\'e~$4$ et d'exposant~$2$
sur un corps $k$, Albert associe une forme quadratique $q$ de
dimension~$6$ et de discriminant\footnote{Pour \'eviter les distinctions
  de cas, nous appelons \emph{discriminant} d'une forme quadratique
  sur un corps de caract\'eristique~$2$ son invariant d'Arf.}
trivial sur $k$, d\'etermin\'ee de mani\`ere unique \`a similitude pr\`es, dont
l'indice de Witt est directement li\'e \`a l'indice de Schur de l'alg\`ebre:
$A$ est d\'eploy\'ee (resp.\ est un corps) si et seulement si $q$ est
hyperbolique (resp.\ est anisotrope), et par cons\'equent $A$ est
d'indice~$2$ si et seulement si $q$ est d'indice~$1$; voir
\cite[\S16.A]{BoI}. Si $L$ est une $k$-alg\`ebre \'etale de dimension~$4$,
les 
\'el\'ements de ${}_2\Br(L/k)$ sont tous repr\'esent\'es par des alg\`ebres
simples centrales de degr\'e~$4$ et d'exposant~$2$ sur $k$. On montre
dans cette section comment calculer une forme d'Albert associ\'ee \`a une
alg\`ebre dont la classe de Brauer est dans ${}_2\Br(L/k)$ \`a partir
d'une repr\'esentation de cette classe comme dans le
corollaire~\ref{cor:Brauer}. Si $k$ est un corps fini, la question est
sans int\'er\^et puisque $\Br(k)=0$. Dans toute cette section, on suppose
donc que $k$ est un corps infini.

Soit $S/C$ l'extension r\'esolvante associ\'ee \`a la $k$-alg\`ebre
$L$. L'alg\`ebre de Lie du tore $T_L$ s'identifie \`a $L$ (avec un crochet
de Lie nul), et la diff\'erentielle du morphisme $\sigma^*\colon T_L\to
T_S$ est une application lin\'eaire $d\sigma^*\colon L\to S$
d\'ecrite dans \cite[(5.55)]{KT} (o\`u l'image de tout $x\in L$ est not\'ee
$\lambda_x$). Rappelons de \cite[Prop.~5.16]{KT} qu'il existe un ouvert de
Zariski non vide de $L$ dont les \'el\'ements $\ell$ satisfont les
propri\'et\'es suivantes:
\begin{enumerate}
\item[(i)]
la $k$-alg\`ebre engendr\'ee par $\ell$ est $L$;
\item[(ii)]
la $C$-alg\`ebre engendr\'ee par $d\sigma^*(\ell)$ est $S$;
\item[(iii)]
la $k$-alg\`ebre engendr\'ee par
$N_{S/C}\bigl(d\sigma^*(\ell)\bigr)$ est $C$.
\end{enumerate}
Fixons un tel \'el\'ement $\ell\in L$ et notons
$c=N_{S/C}\bigl(d\sigma^*(\ell)\bigr)\in C$. Les \'el\'ements $1$,
$c$, $c^2$ forment donc une base de $C$ sur $k$. Choisissons une forme
lin\'eaire $s\colon C\to k$ non nulle telle que $s(1)=s(c)=0$. 
Si $\ker s$ contient un id\'eal non nul de $C$, alors il contient un
idempotent $e\neq0$, $1$. Cet idempotent s'\'ecrit $e=\alpha+\beta c$
pour certains $\alpha\in k$, $\beta\in k^\times$, et l'\'egalit\'e $e^2=e$
entraine que $c$ est racine d'un polyn\^ome de degr\'e~$2$ \`a coefficients
dans $k$. C'est impossible puisque $c$ engendre $C$, donc $\ker s$ ne
contient pas d'id\'eal non nul. On peut alors associer \`a toute forme
quadratique non singuli\`ere $q\colon V\to C$ sur un $C$-module libre
$V$ une forme quadratique non singuli\`ere $s_*(q)\colon V\to k$ d\'efinie
par
\[
s_*(q)(v)=s\bigl(q(v)\bigr) \qquad\text{pour $v\in V$.}
\]
La forme $s_*(q)$ est appel\'ee \emph{transfert} de $q$ (induit par
$s$); voir \cite[\S20.A]{EKM} pour les propri\'et\'es du transfert des
formes quadratiques. 
Si $C$ n'est pas un corps, une forme
quadratique sur $C$ est une collection de formes quadratiques sur
les composantes simples de $C$; cela ne change rien d'essentiel aux
propri\'et\'es du transfert.

Pour
$x\in C^\times$ on consid\`ere en particulier la forme quadratique
$\qf{x}N_{S/C}\colon S\to C$ qui envoie $y\in S$
sur $x\cdot N_{S/C}(y)\in C$.

\begin{thm}
  \label{thm:Albert}
  Pour tout $x\in C^\times$, la forme quadratique
  $s_*(\qf{x}N_{S/C})$ obtenue par transfert de la forme
  quadratique $\qf{x}N_{S/C}$ au moyen de $s$ est une forme
  d'Albert de l'alg\`ebre de degr\'e~$4$ qui repr\'esente
  $\cor_{C/k}(S/C,x)$. 
\end{thm}

La d\'emonstration est pr\'ec\'ed\'ee de deux lemmes.

\begin{lemme}
  \label{lem:hyperb}
  La forme quadratique $s_*(N_{S/C})$ est hyperbolique.
\end{lemme}

\begin{proof}
  Montrons en premier lieu que le sous-espace
  $k+k\cdot d\sigma^*(\ell)\subseteq S$ est totalement isotrope pour la
  forme $s_*(N_{S/C})$. En notant $\Tr_{S/C}$ la trace
  $S\to C$, on a pour $\alpha$, $\beta\in k$
  \[
  N_{S/C}\bigl(\alpha+\beta d\sigma^*(\ell)\bigr) =
  \alpha^2+\alpha\beta \Tr_{S/C}\bigl(d\sigma^*(\ell)\bigr) +
  \beta^2 c.
  \]
  Or, en prenant les diff\'erentielles des morphismes du carr\'e de droite
  de \eqref{diag:4}, on voit que
  \begin{equation}
  \label{eq:trace}
  \Tr_{S/C}\bigl(d\sigma^*(\ell)\bigr)=\Tr_L(\ell)\in k.
  \end{equation}
  (Voir aussi \cite[Prop.~5.16]{KT}.) D\`es lors,
  $N_{S/C}\bigl(\alpha+\beta d\sigma^*(\ell)\bigr)$ est dans le
  $k$-sous-espace de $C$ engendr\'e par $1$ et $c$. Comme $s(1)=s(c)=0$,
  la forme quadratique $s_*(N_{S/C})$ s'annule sur
  $k+k\cdot d\sigma^*(\ell)$.

  De cette observation, il r\'esulte que l'indice de Witt de
  $s_*(N_{S/C})$ est au moins \'egal \`a~$2$, donc
  $s_*(N_{S/C})$ est Witt-\'equivalente \`a une forme quadratique
  binaire. Pour achever de prouver le lemme, il suffit donc de prouver
  que le discriminant $\disc\bigl(s_*(N_{S/C})\bigr)$ est
  trivial. 

  Soit $\sgr_2$ le groupe \`a deux \'el\'ements. On consid\`ere
  $\disc\bigl(s_*(N_{S/C})\bigr)$ comme un \'el\'ement de
  $H^1(k,\sgr_2)$:
  \[
  \disc\bigl(s_*(N_{S/C})\bigr)\in H^1(k,\sgr_2).
  \]
  Pour le calculer, on choisit une forme quadratique hyperbolique
  $q_0\colon S\to C$. Soit $\Orth(q_0)$ son groupe orthogonal. La
  classe d'isom\'etrie de la forme $N_{S/C}$ s'identifie \`a un
  \'el\'ement $\nu$ de l'ensemble de cohomologie $H^1(C,\Orth(q_0))$, et
  son discriminant \`a l'image de $\nu$ dans $H^1(C,\sgr_2)$ sous
  l'application induite par le d\'eterminant (ou, en
  caract\'eristique~$2$, l'invariant de Dickson \cite[(12.12)]{BoI})
  $\delta\colon\Orth(q_0)\to\sgr_2$:
  \[
  \disc(N_{S/C})=\delta(\nu)\in H^1(C,\sgr_2).
  \]
  Toute isom\'etrie de $q_0$ peut \^etre vue comme une isom\'etrie de son
  transfert $s_*(q_0)$, donc on a un morphisme canonique de
  $k$-sch\'emas en groupes $s_\dagger\colon
  R_{C/k}\bigl(\Orth(q_0)\bigr)\to \Orth\bigl(s_*(q_0)\bigr)$. Ce
  morphisme fait commuter le diagramme suivant:
  \[
  \xymatrix@R=5mm{R_{C/k}\bigl(\Orth(q_0)\bigr)
    \ar[r]^{\delta}\ar[d]_{s_\dagger} &
  R_{C/k}(\sgr_2)\ar[d]^{N_{C/k}}\\
  \Orth\bigl(s_*(q_0)\bigr)\ar[r]^{\delta}&\sgr_2
  }
  \]
  On en d\'eduit un diagramme commutatif en cohomologie:
  \[
  \xymatrix@R=5mm{H^1\bigl(k,R_{C/k}(\Orth(q_0))\bigr)
    \ar[r]^{\delta}\ar[d]_{s_\dagger} &
    H^1\bigl(k,R_{C/k}(\sgr_2)\bigr)\ar[d]^{N_{C/k}}\\
  H^1\bigl(k,\Orth(s_*(q_0))\bigr)\ar[r]^{\delta}& H^1(k,\sgr_2)
  }
  \]
  Comme $s_\dagger(\nu)$ repr\'esente la classe d'isom\'etrie de
  $s_*(N_{S/C})$, il en d\'ecoule
  \begin{equation}
  \label{eq:disc1}
  \disc\bigl(s_*(N_{S/C})\bigr) =
  N_{C/k}\bigl(\disc(N_{S/C})\bigr).
  \end{equation}

  Pour conclure, on fait le lien avec une autre notion de
  discriminant: \`a toute $k$-alg\`ebre \'etale $E$ on fait correspondre une
  classe de cohomologie $\xi_E\in H^1(k,\sgr_n)$, o\`u $n=\dim E$ et
  $\sgr_n$ est le groupe des permutations de $n$ lettres (voir
  \cite[(29.9)]{BoI} ou \cite{W}). Le \emph{discriminant} $\disc(E/k)$
  est l'image de $\xi_E$ dans $H^1(k,\sgr_2)$ sous l'application
  d\'eduite de la signature $\sgr_n\to\sgr_2$. Cette notion s'applique
  aussi \`a la $C$-alg\`ebre $S$, et on v\'erifie sans peine que
  \begin{equation}
  \label{eq:disc2}
  \disc(S/C)=\disc(N_{S/C})\in H^1(C,\sgr_2) =
  H^1(k,R_{C/k}(\sgr_2)). 
  \end{equation}
  La formule pour le discriminant d'une tour d'extensions
  \cite[Th.~4]{W} donne
  \begin{equation*}
    \disc(S/k)=N_{C/k}\bigl(\disc(S/C)\bigr).
  \end{equation*}
  Or, la $k$-alg\`ebre $S$ a un discriminant trivial d'apr\`es
  \cite[Prop.~5.1]{KT}, donc
  \begin{equation}
    \label{eq:disc4}
    N_{C/k}\bigl(\disc(S/C)\bigr)=1.
  \end{equation}
  Le lemme d\'ecoule de \eqref{eq:disc1}, \eqref{eq:disc2}
  et \eqref{eq:disc4}.
\end{proof}

Dans l'\'enonc\'e suivant, on d\'esigne par $I^2_q(k)$ le groupe de Witt des
formes quadratiques sur $k$ de discriminant trivial et par $\Cl\colon
I^2_q(k)\to{}_2\Br(k)$ l'homomorphisme induit par l'application qui \`a
chaque forme quadratique associe son alg\`ebre de Clifford. On d\'efinit
de m\^eme le groupe $I^2_q(C)$ et l'homomorphisme $\Cl\colon
I^2_q(C)\to{}_2\Br(C)$. 

\begin{lemme}
  \label{lem:Clif}
  Le diagramme suivant est commutatif:
  \[
  \xymatrix@R=5mm{
  I^2_q(C)\ar[r]^{s_*}\ar[d]_{\Cl} & I^2_q(k)\ar[d]^{\Cl}\\
  {}_2\Br(C)\ar[r]^{\cor_{C/k}} &{}_2\Br(k)
  }
  \]
\end{lemme}

\begin{proof}
  Si la caract\'eristique est diff\'erente de $2$, le groupe $I^2_q(C)$
  est engendr\'e par les $2$-formes de Pfister
  $\langle1,-a\rangle\langle1,-b\rangle$ o\`u $a\in k^\times$ et
  $b\in C^\times$: voir \cite[Lemma~2]{M}. La commutativit\'e du
  diagramme d\'ecoule alors ais\'ement de la r\'eciprocit\'e de Frobenius
  \cite[Prop.~20.2]{EKM} et de la formule de projection. Les m\^emes
  arguments 
  s'appliquent en caract\'eristique~$2$ car $I^2_q(C)$ est engendr\'e
  comme $W(C)$-module
  par les $2$-formes de Pfister quadratiques $[1,a]\langle1,b\rangle$
  et $[1,b]\langle1,a\rangle$ o\`u $a\in k^\times$ et $b\in C^\times$.
\end{proof}

\begin{proof}[D\'emonstration du th\'eor\`eme~\ref{thm:Albert}]
  Vu la caract\'erisation des formes d'Albert dans \cite[(16.3)]{BoI},
  il s'agit de prouver que $s_*(\langle x\rangle\cdot
  N_{S/C})\in I^2_q(k)$ et
  \[
  \Cl\bigl(s_*(\langle x\rangle\cdot N_{S/C})\bigr) =
  \cor_{C/k}(S/C,x) \qquad\text{pour tout $x\in C^\times$.}
  \]
  Du lemme~\ref{lem:hyperb} on d\'eduit l'\'egalit\'e suivante dans le
  groupe de Witt
  des formes quadratiques non d\'eg\'en\'er\'ees de dimension paire:
  \[
  s_*(\langle x\rangle\cdot N_{S/C}) =
  s_*(\langle1,x\rangle\cdot N_{S/C}).
  \]
  Comme $\langle1,x\rangle\cdot N_{S/C}\in I^2_q(C)$, le
  lemme~\ref{lem:Clif} montre d'une part $s_*(\langle x\rangle\cdot
  N_{S/C})\in I^2_q(k)$, et d'autre part
  \[
  \Cl\bigl(s_*(\langle x\rangle\cdot N_{S/C})\bigr) =
  \cor_{C/k}\bigl(\Cl(\langle1,x\rangle\cdot N_{S/C})\bigr).
  \]
  Le th\'eor\`eme r\'esulte alors de l'observation que
  $\Cl(\langle1,x\rangle\cdot N_{S/C})=(S/C,x)$. 
\end{proof}

\`A partir du th\'eor\`eme~\ref{thm:Albert} il est ais\'e de d\'eterminer
quelles sont les alg\`ebres de quaternions sur $k$ qui sont d\'eploy\'ees
par $L$: celles-ci sont Brauer-\'equivalentes \`a des alg\`ebres de
degr\'e~$4$ dont la forme d'Albert est isotrope. Nous allons pr\'eciser la
forme de ces alg\`ebres, en supposant pour simplifier l'expos\'e que la
caract\'eristique est diff\'erente de~$2$, hypoth\`ese qui n'est cependant
nullement essentielle. Nous utiliserons la notation usuelle
$(\alpha,\beta)_k$ pour la $k$-alg\`ebre de quaternions
$(k[\sqrt\alpha]/k,\beta)$. Soient $\ell_1$, $\ell_2$, $\ell_3$,
$\ell_4\in k_s$ les racines du polyn\^ome minimal de $\ell$ sur $k$, et
soit $\rho(X)\in k[X]$ le polyn\^ome unitaire dont les racines dans
$k_s$ sont
\[
{\textstyle\frac14}(\ell_1+\ell_2-\ell_3-\ell_4)^2,\quad
{\textstyle\frac14}(\ell_1-\ell_2+\ell_3-\ell_4)^2,\quad
{\textstyle\frac14}(\ell_1-\ell_2-\ell_3+\ell_4)^2.
\]
Le polyn\^ome $\rho$ est une cubique r\'esolvante du polyn\^ome minimal de
$\ell$: voir \cite[(5.53)]{KT}.

\begin{corol}
  \label{cor:quat}
  Les alg\`ebres de quaternions sur $k$ d\'eploy\'ees par $L$ sont les
  alg\`ebres de quaternions qui peuvent s'\'ecrire sous la forme
  $\bigl(\lambda,-\rho(\lambda)\bigr)_k$ pour un certain $\lambda\in
  k^\times$ tel que $\rho(\lambda)\neq0$.
\end{corol}

\begin{proof}
  Le polyn\^ome minimal sur $C$ de $d\sigma^*(\ell)\in S$ est
  $X^2-\Tr_L(\ell)+c$: voir~\eqref{eq:trace}; on a donc
  \[
  S=C(\sqrt a) \quad\text{avec}\quad
  a={\textstyle\frac14}\bigl(d\sigma^*(\ell)-\gamma(d\sigma^*(\ell))\bigr)^2
  = {\textstyle\frac14}\Tr_L(\ell)^2-c\in C.
  \]
  Soit $Q$ une alg\`ebre de quaternions sur $k$ d\'eploy\'ee par $L$. Si $Q$
  est d\'eploy\'ee, on peut l'\'ecrire sous la forme indiqu\'ee en prenant
  pour $\lambda$ un \'el\'ement de $k^{\times2}$ qui n'est pas racine de
  $\rho$. On peut donc supposer dans la suite que $Q$ n'est pas
  d\'eploy\'ee. D'apr\`es le corollaire~\ref{cor:Brauer} on peut trouver
  $x\in C^\times$ tel que $Q=\cor_{C/k}(S/C,x)$ dans
  $\Br(k)$. Comme l'indice de $Q$ est~$2$, la forme d'Albert de
  $\cor_{C/k}(S/C,x)$ est isotrope, donc le
  th\'eor\`eme~\ref{thm:Albert} montre que la forme quadratique
  $\qf{x}\cdot N_{S/C}$ repr\'esente un \'el\'ement de $\ker s$,
  c'est-\`a-dire un \'el\'ement de la forme $\lambda-\mu a$ avec $\lambda$,
  $\mu\in k$ puisque $\ker s= k+ck=k+ak$. Alors $x\cdot(\lambda-\mu
  a)$ est une norme pour $S/C$, donc
  $(S/C,x)=(S/C,\lambda-\mu a)$ et
  \[
  Q=\cor_{C/k}(a,\lambda-\mu a)_C\qquad\text{pour certains $\lambda$,
    $\mu\in k$ tels que $\lambda-\mu a\in C^\times$.}
  \]
  Si $\mu=0$, alors le corollaire~\ref{cor:Brauer} indique que $Q$ est
  d\'eploy\'ee, contrairement \`a l'hypoth\`ese. Donc $\mu\neq0$ et puisque
  $\cor_{C/k}(S/C,\alpha)=0$ pour tout $\alpha\in k^\times$ on
  peut supposer $\mu=1$. Alors
  \[
  Q=\cor_{C/k}(a,\lambda-a)_C \qquad\text{pour un certain
    $\lambda\in k$ tel que $\lambda-a\in C^\times$.}
  \]
  On a $\lambda\neq0$, sinon $Q$ est d\'eploy\'ee. Comme
  $\cor_{C/k}(a,\lambda)_C=0$ on a aussi
  \[
  Q=\cor_{C/k}(a,\lambda(\lambda-a))_C \qquad\text{pour un certain
    $\lambda\in k^\times$ tel que $\lambda-a\in C^\times$.}
  \]
  Or, en changeant de g\'en\'erateurs quaternioniens on voit que
  \[
  (\lambda,a-\lambda)_C=(a,\lambda(\lambda-a))_C,
  \] 
  donc finalement
  \[
  Q=\cor_{C/k}(\lambda,a-\lambda)_C = (\lambda,
  N_{C/k}(a-\lambda))_k
  \]
  pour un certain $\lambda\in k^\times$ tel que
  $\lambda-a\in C^\times$. 
  Il reste \`a voir que $-\rho(\lambda)=N_{C/k}(a-\lambda)$. Pour
  cela, on observe que les images de $d\sigma^*(\ell)\in S$ sous
  les diff\'erents homomorphismes de $k$-alg\`ebres $S\to k_s$ sont
  les $\ell_i+\ell_j$ avec $1\leq i<j\leq 4$: voir
  \cite[Prop.~5.16]{KT}. Comme
  $a=\frac14\bigl(d\sigma^*(\ell)-\gamma(d\sigma^*(\ell))\bigr)^2$, il
  en r\'esulte que les images de $a$ sous les homomorphismes de
  $k$-alg\`ebres $C\to k$ sont les racines de $\rho$, donc
  $\rho(X)=N_{C/k}(X-a)$ dans $k[X]$. 

  On a ainsi prouv\'e que toute alg\`ebre de quaternions sur $k$ d\'eploy\'ee
  par $L$ est de la forme indiqu\'ee. R\'eciproquement, si
  $Q=\bigl(\lambda,-\rho(\lambda)\bigr)_k$ pour un certain $\lambda\in
  k^\times$ tel que $\rho(\lambda)\neq0$, alors les calculs pr\'ec\'edents
  montrent que $Q=\cor_{C/k}(a,\lambda-a)_C$, donc $Q$ est
  d\'eploy\'ee par $L$ vu le corollaire~\ref{cor:Brauer}.
\end{proof}

Le corollaire~\ref{cor:quat} a \'et\'e prouv\'e par des m\'ethodes diff\'erentes
dans \cite[Cor.~22]{HT}, \cite[Th.~6.2]{HS} et \cite[Cor.~4]{S0} (et
dans \cite[Th.~3.9]{LLT} pour le cas particulier o\`u $L$ est une
$2$-extension). 

\section{Exemples}
\label{sec:ex}

Pour conclure, nous donnons quelques exemples destin\'es \`a illustrer les
constructions et les principaux r\'esultats de ce travail.

Le cas le plus g\'en\'eral est celui o\`u $L$ est un corps, extension
s\'eparable de $k$ dont la cl\^oture galoisienne a pour groupe de Galois
le groupe sym\'etrique $\sgr_4$. Supposons $L\subset k_s$, et d\'esignons
par $M$ la cl\^oture galoisienne de $L$ dans $k_s$. Identifions le
groupe de Galois de $M$ sur $k$ \`a $\sgr_4$ de sorte que $L$ soit le
sous-corps de $M$ fixe sous le groupe des permutations de
$\{2,3,4\}$. Le corps $M$ contient trois sous-corps conjugu\'es
isomorphes \`a $S$ et trois sous-corps conjugu\'es isomorphes \`a $C$:
on peut choisir
\begin{align*}
  S & = M^G&&\text{o\`u\quad
                        $G=\{\Id,(1,2),(3,4),(1,2)(3,4)\}$}\\ 
  \intertext{et}
  C & = M^H&&\text{o\`u\quad $H=G\cup\{(1,3)(2,4), (1,3,2,4),
                    (1,4)(2,3), (1,4,2,3)\}$.}
\end{align*}
Avec ces choix, l'homomorphisme $\sigma^*\colon L^\times\to
S^\times$ envoie $x\in L^\times$ sur $x\cdot(1,2)(x)$ et
l'homomorphisme $\tau^*\colon S^\times\to L^\times$ envoie
$y\in S^\times$ sur $y\cdot(2,3)(y)\cdot(2,4)(y)\in L^\times$.
\medbreak

Consid\'erons ensuite quelques cas particuliers.

\begin{example}[Extensions biquadratiques]
  \label{ex:biquadratic}
  Soit $L$ un corps, extension galoisienne de $k$ de groupe de Galois
  ab\'elien \'el\'ementaire d'ordre~$4$. Soient $L_1$, $L_2$, $L_3$ les
  sous-corps de $L$ qui sont quadratiques sur $k$. Alors 
  \[
  C=k\times k\times k\qquad\text{et}\qquad S=L_1\times L_2\times
  L_3.
  \]
  Les homomorphismes 
  $\sigma^*$ et $\tau^*$ op\`erent comme suit:
  \begin{align*}
    \sigma^*\colon&L^\times\to L_1^\times\times L_2^\times\times
                    L_3^\times &
    &x\mapsto
          \bigl(N_{L/L_1}(x),N_{L/L_2}(x),N_{L/L_3}(x)\bigr),
  \\
  \tau^*\colon&L_1^\times\times L_2^\times\times L_3^\times\to
                L^\times &
    &(y_1,y_2,y_3)\mapsto y_1y_2y_3.
  \end{align*}
  La proposition~\ref{prop:carrenorme} prend la forme suivante:
  \[
  \{x\in k^\times\mid x^2\in N(L/k)\} = k^{\times2}\cdot N(L_1/k)\cdot
  N(L_2/k)\cdot N(L_3/k)
  \]
  et
  \[
  k^{\times2}\cdot N(L/k) = N(L_1/k)\cap N(L_2/k)\cap N(L_3/k).
  \]
  La premi\`ere \'egalit\'e a d\'ej\`a \'et\'e observ\'ee (voir
  \cite[Ex.~5.1, p.~360]{CaF}); la deuxi\`eme traduit le \flqq lemme
  biquadratique\frqq\ bien connu (voir
  \cite[2.13]{EL}, \cite[Lemma~3]{MT} et les r\'ef\'erences qui y sont
  cit\'ees). La remarque~\ref{rem:precis} donne une pr\'ecision
  suppl\'ementaire: si on se donne
  $y_i\in L_i^\times$ pour $i=1$, $2$, $3$ tels que 
  \[
  N_{L_1/k}(y_1)=N_{L_2/k}(y_2)=N_{L_3/k}(y_3),
  \]
  alors on peut trouver $x\in L^\times$ et $\lambda\in k^\times$ tels
  que
  \[
  y_1=\lambda\, N_{L/L_1}(x),\qquad y_2=\lambda\, N_{L/L_2}(x),\qquad
  y_3=\lambda\, N_{L/L_3}(x).
  \]
  Le corollaire \ref{cor:Brauer} \'etablit 
  \[
  {}_2\Br(L/k)=\{
  (L_1/k,x_1)+(L_2/k,x_2)+(L_3/k,x_3)\mid x_1,x_2,x_3\in k^\times\},
  \]
  et montre que l'\'egalit\'e
  \[
  (L_1/k,x_1)+(L_2/k,x_2)+(L_3/k,x_3)=0\qquad\text{dans $\Br(k)$}
  \]
  entraine l'existence d'un \'el\'ement $\lambda\in k^\times$ tel que
  \[
  (L_i/k,x_i)=(L_i/k,\lambda)\qquad\text{pour $i=1$, $2$, $3$.}
  \]
  Ces r\'esultats sont connus: le premier est d\'emontr\'e dans
  \cite[Prop.~5.2]{LLT} et le second est une forme du \emph{``Common
    Slot Theorem''} \cite[Ch.~III, Th.~4.13]{Lam}.
\end{example}

\begin{remm}
  Vu la description de $S$ dans l'exemple~\ref{ex:biquadratic}, le
  tore coflasque associ\'e par Colliot-Th\'el\`ene \cite[\S3]{CT} au tore
  des \'el\'ements de norme~$1$ d'une extension biquadratique est
  $T^1_S$, et le morphisme $\tau^*$ apparait dans la ligne m\'ediane
  de \cite[(3.1)]{CT}. (Lorsque $L$ n'est pas biquadratique, le tore
  $T^1_S$ n'est pas n\'ecessairement coflasque.)
\end{remm}

\begin{example}[Extensions cycliques]
  Supposons que $L$ soit un corps, extension cyclique de $k$, et soit
  $K$ l'unique extension quadratique de $k$ contenue dans $L$. Soit
  encore $\theta$ un g\'en\'erateur du groupe de Galois de $L$ sur $k$. Dans
  ce cas
  \[
  C=k\times K\qquad\text{et}\qquad S=K\times L,
  \]
  et les homomorphismes $\sigma^*$ et $\tau^*$ sont donn\'es comme suit:
  \begin{align*}
    \sigma^*\colon&L^\times\to K^\times\times L^\times &
    &x\mapsto
          \bigl(N_{L/K}(x),x\theta(x)\bigr),
  \\
  \tau^*\colon&K^\times\times L^\times\to
                L^\times &
    &(y_1,y_2)\mapsto y_1y_2\theta^{-1}(y_2).
  \end{align*}
  Par la proposition~\ref{prop:carrenorme} on a
  \[
  \{x\in k^\times\mid x^2\in N(L/k)\} = N(K/k)\cdot N(L/k)
  \text{ et } N(K/k)\cap N(L/K) = k^{\times2}\cdot N(L/k).
  \]
  De plus, vu l'exactitude de la suite inf\'erieure
  de~\eqref{diag:Brauer} en $S^\times$, on voit que si $y_1\in
  K^\times$ et $y_2\in L^\times$ satisfont
  $N_{K/k}(y_1)=N_{L/K}(y_2)$, alors on peut trouver $\lambda\in
  k^\times$ et $x\in L^\times$ tels que $y_1=\lambda N_{L/K}(x)$ et
  $y_2=\lambda x\theta(x)$. Le corollaire~\ref{cor:Brauer} donne
  \[
  {}_2\Br(L/k)=\{(K/k,x_1)+\cor_{K/k}(L/K,x_2)\mid x_1\in k^\times,
  x_2\in K^\times\};
  \]
  de plus, il montre que pour $x_1\in k^\times$ et $x_2\in K^\times$,
  la relation
  \[
  (K/k,x_1)+\cor_{K/k}(L/K,x_2)=0\qquad\text{dans $\Br(k)$}
  \]
  entraine l'existence de $\lambda\in k^\times$ tel que
  \[
  (K/k,x_1)=(K/k,\lambda)\qquad\text{et}\qquad
  (L/K,x_2)=(L/K,\lambda).
  \]
\end{example}

Les deux exemples pr\'ec\'edents sont des cas particuliers du suivant:

\begin{example}[$2$-extensions quartiques]
  \label{ex:2}
  On dit que le corps $L$ est une \emph{$2$-extension quartique} de
  $k$ s'il est une extension s\'eparable de degr\'e~$4$ qui contient une
  extension quadratique $K$ de $k$. Une telle extension est
  caract\'eris\'ee par la condition que l'alg\`ebre $C$ n'est pas un
  corps; dans ce cas $C=k\times\Delta(L)$ o\`u $\Delta(L)$ est la
  $k$-alg\`ebre quadratique discriminante de $L$ (qui correspond \`a
  $\disc(L/k)\in H^1(k,\sgr_2)$ par la correspondance entre
  $H^1(k,\sgr_2)$ et les classes d'isomorphisme de $k$-alg\`ebres
  quadratiques \'etales); voir \cite[Prop.~6.12]{KT}. La $k$-alg\`ebre $L$
  est biquadratique si et seulement si $\Delta(L)=k\times k$. Si l'on
  \'ecarte ce cas (trait\'e en d\'etail dans
  l'exemple~\ref{ex:biquadratic}), l'extension quadratique $K$ est
  unique, et on associe canoniquement \`a $L$ (\`a isomorphisme pr\`es) une
  $k$-alg\`ebre \'etale 
  quartique $\check{L}$ \flqq duale\frqq\ de $L$ (voir
  \cite[(6.19)]{KT}). L'alg\`ebre $\check{L}$ contient $\Delta(L)$ et
  satisfait
  \[
  \Delta(\check{L})=K\qquad\text{et}\qquad S=K\times\check{L}.
  \]
  (Le cas particulier o\`u $L$ est une extension cyclique de $k$ est
  celui o\`u $\check{L}=L$.) La proposition~\ref{prop:carrenorme} donne
  \[
  \{x\in k^\times\mid x^2\in N(L/k)\} = N(K/k)\cdot N(\check{L}/k)
  \]
  et
  \[
  N(K/k)\cap N(\check{L}/\Delta(L))=k^{\times2}\cdot N(L/k).
  \]
  D'autre part, le corollaire~\ref{cor:Brauer} montre
  \[
  {}_2\Br(L/k)=\{(K/k,x_1)+\cor_{\Delta(L)/k}(\check{L}/\Delta(L),
  x_2) \mid x_1\in k^\times, x_2\in\Delta(L)^\times\}.
  \]
  De plus, si $x_1\in k^\times$ et $x_2\in \Delta(L)^\times$ satisfont
  \[
  (K/k,x_1)+\cor_{\Delta(L)/k}(\check{L}/\Delta(L),
  x_2)=0\qquad\text{dans $\Br(k)$},
  \]
  alors il existe $\lambda\in k^\times$ tel que
  \[
  (K/k,x_1)=(K/k,\lambda) \qquad\text{et}\qquad
  (\check{L}/\Delta(L),x_2) = (\check{L}/\Delta(L),\lambda).
  \]
\end{example}

Dans notre dernier exemple, la $k$-alg\`ebre quartique $L$ n'est pas un
corps: 

\begin{example}
  \label{ex:C2xC2}
  Soit $L=K_1\times K_2$ o\`u $K_1$ et $K_2$ sont des extensions
  quadratiques s\'eparables de $k$. On pose $M=K_1\otimes_kK_2$ et on
  d\'esigne par $K_0$ le produit de $K_1$ et $K_2$ dans le groupe des
  extensions quadratiques s\'eparables de $k$ (qui est isomorphe \`a
  $H^1(k,\sgr_2)$; voir \cite[(4.29)]{KT}). On a alors
  \[
  S=(k\times k)\times M \qquad\text{et}\qquad
  C= k\times K_0.
  \]
  Les homomorphismes $\sigma^*$ et $\tau^*$ op\`erent comme suit:
  \begin{align*}
    \sigma^*\colon&K_1^\times\times K_2^\times\to k^\times\times
                    k^\times\times M^\times &&
    (x_1,x_2)\mapsto(N_{K_1/k}(x_1), N_{K_2/k}(x_2),
                    x_1\otimes x_2),
  \\
    \tau^*\colon&k^\times\times k^\times\times M^\times\to
                  K_1^\times\times K_2^\times &&
    (\lambda_1,\lambda_2,y)\mapsto \bigl(\lambda_1 N_{M/K_1}(y), 
    \lambda_2 N_{M/K_2}(y)\bigr).
  \end{align*}
  En l'occurrence, $N(S/k)=k^\times$ et $N(L/k)=N(K_1/k)\cdot
  N(K_2/k)$, donc $k^{\times2}\subseteq N(L/k)$. La
  proposition~\ref{prop:carrenorme} donne simplement
  \[
  \{x\in k^\times\mid x^2\in N(K_1/k)\cdot N(K_2/k)\}=k^\times
  \text{ et }
  N(K_1/k)\cdot N(K_2/k)=N(M/K_0)\cap k^\times.
  \]
  De la remarque~\ref{rem:precis}, on tire l'information plus pr\'ecise
  suivante: si 
  $\lambda_1$, $\lambda_2\in k^\times$ et $m\in M^\times$ sont tels
  que $N_{M/K_0}(m)=\lambda_1\lambda_2$, alors on peut trouver $x_1\in
  K_1^\times$, $x_2\in K_2^\times$ et $\mu\in k^\times$ tels que
  \[
  \lambda_1=\mu\, N_{K_1/k}(x_1),\quad \lambda_2=\mu\,
  N_{K_2/k}(x_2), \quad m=\mu\, x_1\otimes x_2.
  \]
  Le corollaire~\ref{cor:Brauer} donne
  \[
  \Br(K_1/k)\cap\Br(K_2/k)=\{\cor_{K_0/k}(M/K_0,x)\mid x\in
  K_0^\times\}
  \]
  et montre que le noyau de la corestriction $\cor_{K_0/k}\colon
  \Br(M/K_0)\to\Br(k)$ est form\'e des classes d'alg\`ebres de quaternions
  de la forme $(M/K_0,\lambda)$ o\`u $\lambda\in k^\times$.
\end{example}

\begin{remm}
  Soit $L$ une alg\`ebre \'etale quartique sur un corps global
  $k$. Supposons que l'alg\`ebre cubique r\'esolvante $C$ ne soit pas un
  corps; elle a donc au moins un facteur isomorphe \`a $k$. Si $C\simeq
  k\times C'$ o\`u $C'$ est une $k$-alg\`ebre \'etale quadratique, alors
  $T^1_C\simeq T_{C'}$, donc $H^1(k,T^1_C)=1$ par le th\'eor\`eme~90 de
  Hilbert et $\Sha^2(k,T^1_C)=1$ par le th\'eor\`eme de
  Brauer--Hasse--Noether--Albert. D\`es lors, la suite exacte
  \eqref{exseq:2} donne un isomorphisme induit par le morphisme
  $\tau^*_J$:
  \begin{equation}
    \label{eq:PH}
    \Sha^1(k,T^1_S)\simeq \Sha^1(k,T^1_L).
  \end{equation}

  Si par exemple $L$ est un corps, extension biquadratique de $k$, et
  que $L_1$, $L_2$, $L_3$ sont les sous-corps de $L$ qui sont
  quadratiques sur $k$, comme dans l'exemple~\ref{ex:biquadratic},
  alors $\Sha^1(k,T^1_S)$ mesure l'obstruction au principe de Hasse
  pour les \'equations de la forme
  \begin{equation}
    \label{eq:multi}
    N_{L_1/k}(y_1)\, N_{L_2/k}(y_2)\, N_{L_3/k}(y_3) = a,
  \end{equation}
  appel\'ees \emph{\'equations multinormes}. Vu~\eqref{eq:PH}, le principe
  de Hasse vaut pour l'\'equation~\eqref{eq:multi} si et seulement s'il
  vaut pour l'\'equation $N_{L/k}(x)=a^2$. On sait que ce n'est pas
  toujours le cas: voir \cite[Ex.~5, p.~360]{CaF},
  \cite[Ex.~2]{PoRa}.

  Si $L$ est un corps, $2$-extension quartique mais non biquadratique
  de $k$, comme dans l'exemple~\ref{ex:2}, alors $S$ est un produit
  direct de deux corps, $S\simeq K\times \check{L}$. Dans ce cas, les
  deux membres de \eqref{eq:PH} sont triviaux: cela d\'ecoule d'un
  r\'esultat de Bayer-Fluckiger--Lee--Parimala
  \cite[Prop.~2.3]{BLP} sur les \'equations multinormes pour les
  produits de deux corps et d'un th\'eor\`eme de Bartels
  \cite[Satz~1]{Bar} sur les normes d'extensions dont le groupe de
  Galois est di\'edral. 

  Enfin, si $L=K_1\times K_2$ o\`u $K_1$ et $K_2$ sont deux corps,
  extensions quadratiques de $k$, comme dans l'exemple~\ref{ex:C2xC2},
  alors $\Sha^1(k,T^1_S)=1$ puisque
  $N(S/k)=k^\times$. Vu~\eqref{eq:PH}, le principe de Hasse vaut pour
  les \'equations multinormes de la forme
  \[
  N_{K_1/k}(x_1)\, N_{K_2/k}(x_2)=a.
  \]
  C'est un cas particulier d'un r\'esultat de H\"urlimann
  \cite[Prop.~3.3]{Hu}.
\end{remm}

\bibliography{Quartic}

\begin{flushright}
\begin{tabular}{l}
Universit\'e catholique de Louvain\\ 
Institut ICTEAM
\\ 
Avenue G.~Lema\^{\i}tre 4, Boite~L4.05.01\\
B-1348 Louvain-la-Neuve, Belgique\\
\textit{M\'el.:} \verb+Jean-Pierre.Tignol@uclouvain.be+
\end{tabular}
\end{flushright}
\end{document}